\documentclass[a4paper]{amsart}
\usepackage{amssymb,amsmath,enumerate}
\usepackage{txfonts}
\newcommand{\R}{\mathbb{R}}
\newcommand{\Z}{\mathbb{Z}}
\newcommand{\C}{\mathbb{C}}
\newcommand{\N}{\mathbb{N}}
\newcommand{\co}{\colon\thinspace}

\theoremstyle{plain}
\newtheorem{theorem}{Theorem}[section]
\newtheorem{corollary}[theorem]{Corollary}
\newtheorem{proposition}[theorem]{Proposition}

\theoremstyle{definition}
\newtheorem{definition}[theorem]{Definition}
\newtheorem{remark}[theorem]{Remark}
\title{Generic immersions and totally real embeddings}
\author{Naohiko Kasuya \and Masamichi Takase}
\address{Naohiko Kasuya: 
Department of Mathematics, Kyoto Sangyo University,
Kamigamo-Motoyama, Kita-ku, Kyoto, 603-8555, Japan.}
\email{nkasuya@cc.kyoto-su.ac.jp}%
\address{Masamichi Takase: 
Faculty of Science and Technology,
Seikei University, 3-3-1 Kichijoji-kitamachi, Musashino,
Tokyo 180-8633, Japan.} 
\email{mtakase@st.seikei.ac.jp}%
\date{}
\dedicatory{Dedicated to Professor Takashi Nishimura on the occasion of his 60th birthday}
\begin{document}\sloppy
\maketitle
\begin{abstract}
We show that, for a closed orientable $n$-manifold, with $n$ 
not congruent to $3$ modulo $4$, 
the existence of a CR-regular embedding into complex $(n-1)$-space 
ensures the existence of a totally real embedding into complex $n$-space. 
This implies that a closed orientable $(4k+1)$-manifold 
with non-vanishing Kervaire semi-characteristic possesses 
no CR-regular embedding into complex $4k$-space. 
We also pay special attention to the cases 
of CR-regular embeddings of spheres and 
of simply-connected $5$-manifolds. 
\end{abstract}
\section{Introduction}\label{sect:introduction}
Let $f\co M^n \to \mathbb{C}^q$ be a smooth immersion 
of a $n$-dimensional manifold into $\mathbb{C}^q$. 
Then we see that, for a point $x\in M^n$ and the standard 
complex structure $J$ on $\mathbb{C}^q$, 
\[
\mathrm{dim}_\mathbb{C}(df_x(T_xM^n)\cap Jdf_x(T_xM^n))\geqq n-q. 
\]
The point $x$ is said to be \textit{CR-regular} if the equality holds 
(and \textit{CR-singular} otherwise) in the above inequality. 
A CR-regular point (resp. a CR-singular point) is called 
\textit{a $\R\C$-regular point} (resp. \textit{a $\R\C$-singular point}) 
in \cite{MR0314066}. 
The immersion $f$ is said to be \textit{a CR-regular immersion} if 
it has no CR-singular point. 
A CR-regular immersion $f\co M^n \to \mathbb{C}^q$ 
is also called \textit{a generic immersion} 
(e.\,g., see Jacobowitz and Landweber \cite{MR2304588}), 
and when $n=q$ 
it has yet another name, \textit{a totally real immersion}. 

We will work in the smooth category; throughout the paper 
all manifolds and immersions are supposed to be differentiable 
of class $C^\infty$. 
Manifolds are further supposed to be connected. 

The following is our main theorem. 

\begin{theorem}\label{thm:main}
Let $M^n$ be a closed orientable $n$-manifold with $n>3$. 
Then, 
\begin{enumerate}[(a)]
\item
when $n\equiv0\bmod2$, 
$M^n$ admits a totally real embedding into $\C^n$ 
if and only if it admits a CR-regular immersion into $\C^{n-1}$; 
\item
when $n\equiv1\bmod4$ and $w_2w_{n-2}[M^n]=0$, 
$M^n$ admits a totally real embedding into $\C^n$ 
if and only if it admits a CR-regular immersion into $\C^{n-1}$; 
\item
when $n\equiv1\bmod4$ and $w_2w_{n-2}[M^n]\ne0$, 
if $M^n$ admits a totally real embedding into $\C^n$ then 
it does not admit a CR-regular immersion into $\C^{n-1}$. 
\end{enumerate}
\end{theorem}

We will see later that the condition 
$w_2w_{n-2}[M^n]=0$ (vanishing de Rham invariant) 
above 
can be replaced with $\bar{w}_2\bar{w}_{n-2}[M^n]=0$ 
(see Remark~\ref{rem:normal}). 
Furthermore, the normal bundle of an embedding of a closed 
orientable $n$-manifold into a Euclidean space necessarily 
has trivial Euler class 
(see \cite[Corollary~11.4]{MR0440554}). 
Hence a closed 
orientable $n$-manifold $M^n$ embeddable in $\R^{2n-2}$ 
must have trivial $\bar{w}_{n-2}(M^n)$. 
Therefore, we have the following. 

\begin{corollary}\label{cor:main}
If a closed orientable $n$-manifold, with $n\not\equiv3\bmod4$, 
admits a CR-regular embedding into $\C^{n-1}$, 
then it admits a totally real embedding into $\C^n$. 
\end{corollary}


Note that the converse of Corollary~\ref{cor:main} does not 
hold in general (see Remark~\ref{rem:wu}). 


The proof of Theorem~\ref{thm:main} will be given 
in \S\ref{sect:proof}. 
We will discuss CR-regular immersions 
of simply-connected $5$-manifolds into $\C^4$ 
in \S\ref{sect:5dim}, and CR-regular 
embeddings of spheres in more general codimensions
in \S\ref{sect:spheres}. 

\section{Background}
\subsection{Totally real embeddings}
The problem of determining when a manifold admits 
a totally real immersion or embedding 
has been extensively studied by many authors 
including Gromov \cite{MR0420697, MR0413206,MR864505}, 
Forstneri{\v{c}} \cite{MR880125} and 
Audin \cite{MR966952}. 

First, as for totally real immersions, 
the following theorem 
due to Gromov \cite{MR0420697} and Lees \cite{MR0410764} is 
fundamental. 
(see also \cite{MR1909245}). 

\begin{theorem}[Gromov \cite{MR0420697} and Lees \cite{MR0410764}]\label{thm:gromov}
An $n$-manifold 
$M^n$ admits a totally real immersion into $\C^n$
if and only if 
the complexified tangent bundle $\C TM^n$ is trivial. 
\end{theorem}

This is called the $h$-principle for totally real immersions. 
For totally real embeddings, 
Gromov \cite{MR864505} and Forstneri{\v{c}} \cite{MR880125} have proved the following $h$-principle. 

\begin{theorem}[Gromov \cite{MR864505} and Forstneri{\v{c}} \cite{MR880125}]
Let $M^n$ be a closed orientable $n$-manifold with $n\geq 3$. 
Then, $M^n$ admits a totally real embedding into $\C^n$ if and only if 
it admits a totally real immersion into $\C^n$ which is 
regularly homotopic to an embedding. 
\end{theorem}

Using these $h$-principles, 
Audin \cite{MR966952} has given a necessary and sufficient condition 
for the existence of totally real embeddings. 
In order to state her results, 
we need the notion of semi-characteristic. 

\begin{definition}[Kervaire {\cite{MR0086302}} and Lusztig--Milnor--Peterson {\cite{MR0246308}}]
Let $F$ be either the field $\R$ or $\Z/2\Z$. 
For a closed orientable $(2r+1)$-manifold $M$, 
\textit{the semi-characteristic of $M$ with respect to a coefficient field $F$} is 
defined to be 
\[
\hat\chi_F(M)=\sum_{i=0}^r\mathrm{rank}\thinspace H_i(M;F)\in\Z/2\Z
\]
taken as an integer modulo $2$. 
We call $\hat\chi_{\R}(M)$ \textit{the real semi-characteristic} 
and $\hat\chi_{\Z/2\Z}(M)$ \textit{the Kervaire semi-characteristic}, respectively.  
\end{definition}

\begin{theorem}[Audin \cite{MR966952}]\label{thm:audin}
A closed orientable $n$-manifold $M^n$ 
admits a totally real embedding into $\C^n$
if and only if 
\begin{enumerate}[(a)]
\item
when $n\equiv0\bmod2$, 
$\C TM^n$ is trivial and 
the Euler characteristic 
$\chi(M^n)$ vanishes;
\item
when $n\equiv1\bmod4$, 
$\C TM^n$ is trivial and 
the Kervaire semi-characteristic 
$\hat\chi_{\mathbb{Z}/2\mathbb{Z}}(M^n)$ vanishes. 
\end{enumerate}
\end{theorem}

Note that Audin \cite{MR966952} has given some 
partial results also in the case where $n\equiv3\bmod8$.

\subsection{CR-regular immersions and embeddings}
CR-regular immersions and 
embeddings also have been 
studied by many authors from various viewpoints 
(see e.\,g., Cartan \cite{MR1553196}, 
Tanaka \cite{MR0145555,MR0221418}, 
Wells \cite{MR0241690,MR0237823},  
Lai \cite{MR0314066},
Jacobowitz and Landweber \cite{MR2304588}, 
Ho, Jacobowitz and Landweber \cite{MR2946116}, 
Slapar \cite{MR3345510}, 
Torres \cite{MR3503716, MR3777134}). 
Our previous paper \cite{MR3739200} sheds 
some new light on a relation between CR-singularity 
and differential topology. 

In the recent paper \cite{MR2304588} by 
Jacobowitz and Landweber the following generalisation of 
Gromov--Lees Theorem (Theorem~\ref{thm:gromov}) has been shown. 

\begin{theorem}[Jacobowitz and Landweber \cite{MR2304588}]~\label{thm:JL}
Let $n<2q$. 
Then, an $n$-manifold $M^n$ admits a CR-regular immersion into $\C^q$ 
if and only if the complexified tangent bundle $\C TM^n$ splits as 
\[
\C TM^n=A\oplus B, 
\]
where $A$ is the trivial complex vector bundle of rank $q$, 
and $B$ is a complex vector bundle of rank $n-q$ with $B\cap \overline{B}=\left\{0 \right\}$. 
\end{theorem}

The ``only if'' part of Theorem~\ref{thm:JL} is straightforward: 
if $M^n$ admits a CR-regular immersion into $\C^q$ then we can find 
inside $TM^n$ 
a suitable complex bundle of rank $n-q$, whose 
fibre at $x\in M^n$ is mapped onto 
$df_x(T_xM^n)\cap Jdf_x(T_xM^n)$ via $df_x$ 
(see the argument in the proof of Proposition~\ref{thm:kasuya}). 
The converse is the essence of Theorem~\ref{thm:JL}, which 
clarifies the condition for the existence of CR-regular immersions. 
The next goal is to write down the condition for the existence of CR-regular embeddings. 
In particular, the case of CR-regular embeddings of $n$-manifolds $M^n$ into $\C^{n-1}$, $n\geq3$, 
seems interesting and has been extensively studied in recent years. 

When $n=3$, each immersion or embedding of $3$-manifolds in $\C^2$ is 
CR-regular. 
When $n=4$, Slapar \cite{MR3345510} has shown that 
a closed orientable $4$-manifold $M^4$ admits a CR-regular immersion in $\C^3$ 
if and only if the first Pontryagin class $p_1(M^4)$ and the Euler characteristic $\chi (M^4)$ vanish, 
and admits a CR-regular embedding if and only if in addition the second Stiefel--Whitney class $w_2(M^4)$ vanishes. 

In the $5$-dimensional case, Jacobowitz and Landweber 
\cite[p.\,163]{MR2304588} posed a problem: 
``\textit{
one could ask if every $M^5$ which
admits an embedding into $\C^4$ and satisfies the necessary condition for generic
immersions given by Theorem 1.2 has a generic embedding into $\C^4$ or if the
Kervaire semi-characteristic is still restricted.}''
In the present paper we give an answer to this question. 
Our main theorem, in fact, claims that vanishing Kervaire semi-characteristic 
is a necessary condition for the existence of 
``a generic embedding'' into $\C^{4k}$ for a closed 
orientable $(4k+1)$-manifold 
(see Corollary~\ref{cor:kervaire}). 

For the $5$-dimensional sphere $S^5$, 
Elgindi has already proved in \cite{1607.02384} the non-existence of 
a CR-regular embedding of $S^5$ into $\C^4$. 
We show that the proof found in \cite{1607.02384} can be 
applied to prove the following. 

\begin{proposition}\label{thm:kasuya}
If a closed orientable $n$-manifold $M^n$ admits a CR-regular 
immersion into $\C^{n-1}$ with trivial normal bundle, 
then $M^n$ is parallelisable. 
\end{proposition}

\begin{proof}
Suppose that there exists 
a CR-regular immersion of $M^n$ into $\C^{n-1}$ 
with normal bundle $\mathcal{N}$. 
Then, in each $T_xM^n$, 
we can take such a complex line that is mapped 
onto  $df_x(T_xM^n)\cap Jdf_x(T_xM^n)$ through $df_x$. 
Thus we have a complex line bundle $\mathcal{L}$ over $M^n$ 
and a decomposition 
\[
TM^n\cong\mathcal{L}_{\mathbb{R}}\oplus\mathcal{M}, 
\]
where $\mathcal{L}_{\mathbb{R}}$ is the underlying real bundle 
of $\mathcal{L}$ and $\mathcal{M}$ is the complementary bundle. 
Therefore we have 
\[
\mathcal{L}_{\mathbb{R}}\oplus\mathcal{M}\oplus\mathcal{N}
\cong f^*T\R^{2n-2}=f^*T\C^{n-1}. 
\]
Since $\mathcal{M}$ and $\mathcal{N}$, viewed in $T\C^{n-1}$, 
should be mapped isomorphically onto each other via $J$, 
we see that $\mathcal{M}\cong\mathcal{N}$, 
more precisely 
\[
\mathcal{M}\oplus\mathcal{N}\cong\mathcal{M}\otimes\C. 
\]
Hence we have the bundle isomorphism  
\[
\mathcal{L}\oplus(\mathcal{M}\otimes\C)\cong f^*T\C^{n-1}. 
\]
It follows from $w_1(\mathcal{M})=w_1(M)=0$ that 
$c_1(\mathcal{M}\otimes\C)=0$ (see \cite[Problem~15-D]{MR0440554}). 
Therefore the complex line bundle 
$\mathcal{L}$ has trivial first Chern class 
$c_1(\mathcal{L})$ and hence is trivial. 

Now since $\mathcal{N}$ is assumed to be trivial, 
$\mathcal{M}$ is trivial. 
Consequently, $TM^n\cong\mathcal{L}_{\mathbb{R}}\oplus\mathcal{M}$ 
is trivial, that is, $M^n$ is parallelisable. 
\end{proof}


\begin{remark}
Torres \cite{MR3503716} has studied CR-regular immersions and embeddings of orientable $6$-manifolds
into $\C^4$. Another interesting aspect proved in \cite{MR3777134}, based on Slapar \cite{MR3096869}
and Di Scala--Kasuya--Zuddas \cite{MR3453880}, is that an almost-complex $2n$-manifold $M^{2n}$
admits a pseudo-holomorphic embedding into $\R^{2n+2}$ endowed with a suitable almost-complex structure 
if and only if it admits a CR-regular embedding into $\C^{n+1}$.
\end{remark}

\section{The proof of Theorem~\ref{thm:main} and discussion}\label{sect:proof}
We first prove the following observation. 
%
\begin{proposition}\label{prop:imm}
A closed orientable 
$n$-manifold $M^n$ with $n\geq 3$ admits 
a CR-regular immersion into $\C^{n-1}$ if and only if 
the complexified tangent bundle $\C TM^n$ is trivial 
and $M^n$ has an oriented $2$-plane field 
with trivial Euler class. 
\end{proposition}

\begin{proof}[{Proof of Proposition~\ref{prop:imm}}]
Suppose that there exists 
a CR-regular immersion of $M^n$ into $\C^{n-1}$ 
with normal bundle $\mathcal{N}$. 
Then, by exactly the same argument as in 
the proof of Proposition~\ref{thm:kasuya}, 
for a trivial complex line bundle $\mathcal{L}$ over $M^n$ 
we have the decomposition 
\begin{align}
TM^n\cong\mathcal{L}_{\mathbb{R}}\oplus\mathcal{M}, \label{TM}
\end{align}
where $\mathcal{L}_{\mathbb{R}}$ is the underlying real bundle 
of $\mathcal{L}$ and $\mathcal{M}$ is isomorphic to $\mathcal{N}$. 

From the bundle isomorphism \eqref{TM}, we see that 
the complexified tangent bundle is described as 
\begin{align*}
\C TM^n&\cong\mathcal{L}\oplus\overline{\mathcal{L}}\oplus(\mathcal{M}\otimes\C)\\
&\cong f^*T\C^{n-1}\oplus\overline{\mathcal{L}} 
\end{align*}
(see \cite[Lemma~15.4, p.\,176]{MR0440554}), which is trivial since 
$\overline{\mathcal{L}}$ is trivial. 

Conversely, suppose that 
the complexified tangent bundle $\C TM^n$ is trivial 
and $M^n$ has an oriented $2$-plane field 
with trivial Euler class. 
Then the tangent bundle $TM$ splits as 
\[
TM=P\oplus Q, 
\]
where $P$ is a trivial $2$-plane bundle 
and $Q$ is its complementary bundle. 
Then, the complexified tangent bundle splits as 
\begin{align*}
\C TM^n&=(P\otimes \C)\oplus (Q\otimes \C)\\
&=\epsilon \oplus \epsilon \oplus (Q\otimes \C), 
\end{align*}
where $\epsilon $ is the trivial complex line bundle. 
Thus we see that the complex vector bundle $\epsilon \oplus (Q\otimes \C)$ 
of rank $n-1$ is stably trivial, and actually 
trivial by the dimensional reason \cite[Lemma~1.2]{MR2304588}. 
Therefore, by Theorem~\ref{thm:JL} (putting $A=\epsilon \oplus (Q\otimes \C)$), 
$M^n$ admits a CR-regular immersion into $\C^{n-1}$. 
\end{proof}

The problem of determining when an orientable manifold has an 
oriented $2$-plane field has been solved due to 
Atiyah--Dupont \cite{MR0451256} 
and a series of papers by Thomas (see \cite{MR0287567}). 
We recall here the following special case where 
the Euler class of the $2$-plane bundle 
determined by the $2$-plane field is trivial. 

\begin{theorem}[Atiyah--Dupont~\cite{MR0451256} and Thomas~\cite{MR0287567}]\label{thm:AD}
A closed orientable smooth $n$-manifold $M^n$ has 
an oriented $2$-plane field 
with trivial Euler class if and only if 
either of the following holds: 
\begin{enumerate}[(a)]
\item
$n\equiv0\bmod4$, 
$\chi (M^n)=0$ and the signature 
$\sigma(M^n)$ of $M^n$ 
is congruent to zero modulo $4$; 
\item
$n\equiv1\bmod4$ and 
$\hat\chi_{\R}(M^n)=0$; 
\item
$n\equiv2\bmod4$ and 
$\chi (M^n)=0$; 
\item
$n\equiv3\bmod4$. 
\end{enumerate}
\end{theorem}

By combining Proposition~\ref{prop:imm} with Theorem~\ref{thm:AD}, 
we obtain Corollary~\ref{cor:imm} below. 
Compare it with Theorem~\ref{thm:audin} (Audin). 
Note that if $\C TM^n$ is trivial 
then all the Pontryagin classes of $M^n$ vanish 
and therefore we have $\sigma(M^n)=0$. 

\begin{corollary}\label{cor:imm}
A closed orientable 
$n$-manifold $M^n$ admits 
a CR-regular immersion into $\C^{n-1}$ if and only if 
\begin{enumerate}[(a)]
\item
when $n\equiv0\bmod2$, $\C TM^n$ is trivial and 
$\chi (M^n)=0$; 
\item
when $n\equiv1\bmod4$, $\C TM^n$ is trivial and 
$\hat\chi_{\R}(M^n)=0$; 
\item
when $n\equiv3\bmod4$, $\C TM^n$ is trivial. 
\end{enumerate}
\end{corollary}

Now we prove our main theorem. 

\begin{proof}[{Proof of Theorem~\ref{thm:main}}]
Theorem~\ref{thm:main}~(a) is just a combination of 
Theorem~\ref{thm:audin}~(a) and Corollary~\ref{cor:imm}~(a). 

Suppose that $n\equiv1\bmod4$. 
Then, according to Lusztig--Milnor--Peterson's 
formula \cite{MR0246308} 
(see also \cite[Remark on p.\,16]{MR0263102}) 
and Remark~\ref{rem:normal} below, 
we have 
\[
\hat\chi_{\mathbb{R}}(M^n)-\hat\chi_{\mathbb{Z}/2\mathbb{Z}}(M^n)
=w_2w_{n-2}[M^n]. 
\]
Therefore, if $w_2w_{n-2}[M^n]=0$ then 
$\hat\chi_{\mathbb{R}}(M^n)$ 
and $\hat\chi_{\mathbb{Z}/2\mathbb{Z}}(M^n)$ 
have the same parity, so that 
Theorem~\ref{thm:audin}~(b) and Corollary~\ref{cor:imm}~(b) 
prove the Theorem~\ref{thm:main}~(b). 
On the contrary, if $w_2w_{n-2}[M^n]\ne0$ then 
either one of $\hat\chi_{\mathbb{R}}(M^n)$ 
and $\hat\chi_{\mathbb{Z}/2\mathbb{Z}}(M^n)$ 
must be non-zero and the other zero; 
Theorem~\ref{thm:audin}~(b) and Corollary~\ref{cor:imm}~(b) 
then imply Theorem~\ref{thm:main}~(c). 
\end{proof}

\begin{remark}\label{rem:normal}
In view of \eqref{TM} in the proof of Proposition~\ref{prop:imm}, 
we see that for a closed orientable 
$n$-manifold $M^n$ admitting 
a CR-regular immersion into $\C^{n-1}$ and each $i$, 
\begin{align*}
w_i(M^n)
=w_i(\mathcal{L}_{\mathbb{R}}\oplus\mathcal{M})
=w_i(\mathcal{M})
=w_i(\mathcal{N})
=\bar{w}_i(M^n). \label{SW}
\end{align*}
Hence 
the square of each Stiefel--Whitney class of 
$M^n$ vanishes. 
This is also seen from 
the triviality of $\C TM^n\cong TM\oplus TM$ 
(see \cite{MR966952}). 
Therefore, the condition on the Stiefel--Whitney number 
in Theorem~\ref{thm:main} can be replaced with 
that on the normal Stiefel--Whitney number. 
\end{remark}

As mentioned in \S\ref{sect:introduction}, 
we have $\bar{w}_{n-2}(M^n)=0$ for a closed 
orientable $n$-manifold $M^n$ embeddable in $\R^{2n-2}$, 
and hence Corollary~\ref{cor:main} follows in view of Remark~\ref{rem:normal}. 
As a consequence of Theorem~\ref{thm:main} 
and Corollary~\ref{cor:main} 
we have the following. 

\begin{corollary}\label{cor:kervaire}
If a closed orientable $n$-manifold, with $n\equiv1\bmod4$
(resp.\ $n\equiv0\bmod2$), 
admits a CR-regular embedding into $\C^{n-1}$, 
then it has trivial Kervaire semi-characteristic 
(resp.\ trivial Euler characteristic). 
\end{corollary}

Note that for $n\equiv0\bmod2$ 
we will have the more general Corollary~\ref{cor:even-dim}. 

\begin{remark}\label{rem:span}
Let $M^n$ be a closed orientable 
$n$-manifold $M^n$ which admits 
a CR-regular immersion into $\C^{n-1}$. 
Then, we see from \eqref{TM} that 
$\mathop{\mathrm{span}}M^n\geq2$, where 
$\mathop{\mathrm{span}}M^n$ stands for 
the maximum number of everywhere linearly
independent vector fields on $M^n$. 
\end{remark}

\begin{remark}\label{rem:mahowald}
Let $M^n$ be a closed orientable $n$-manifold and $n>4$. 
When $n$ is odd $M^n$ can be immersed in $\R^{2n-2}$; 
when $n$ is even $M^n$ can be 
be immersed in $\R^{2n-2}$ if and only if $\bar{w}_2\bar{w}_{n-2}[M^n]=0$ 
 (\cite[Theorem~7.1.1]{MR0157386}). 

If $n\equiv3\bmod4$ then $\bar{w}_{n-2}(M^n)$ always vanishes 
(\cite{MR0035992}, see \cite[Thoerem~III]{MR0111053})
and hence $\bar{w}_2\bar{w}_{n-2}[M^n]=0$. 
This is not the case where $n\equiv1\bmod4$; e.\,g., 
the $5$-dimensional Dold manifold $P(1,2)$ can be 
immersed in $\R^8$ and has non-trivial $\bar{w}_2\bar{w}_{3}$. 
Note that $P(1,2)$ does not admit 
a CR-regular immersion into $\C^4$ since its span equals $1$ 
(see Remark~\ref{rem:span}). 
\end{remark}

\begin{remark}\label{rem:wu}
The Wu manifold 
$X_{\infty}=SU(3)/SO(3)$ \cite{MR0035992} 
is a closed simply-connected $5$-manifold 
with $w_2w_3[X_{\infty}]\ne0$ \cite[p.\,80]{MR0061823}. 
It appears in Barden's list \cite{MR0184241}
as one of the ``building blocks'' of simply-connected $5$-manifolds 
(see \S\ref{sect:5dim}), and has been proved to admit 
no embedding into $\R^8$ \cite[Theorem~2.5]{MR0184241}. 
On the other hand, Audin \cite[Proposition~0.8]{MR966952} 
has shown that
$X_{\infty}$ has trivial Kervaire semi-characteristic 
and admits a totally real embedding into $\C^5$. 
This implies that the converse of 
Corollary~\ref{cor:main} does not hold in general. 
We may ask the following question: 
\textit{if a closed orientable $n$-manifold, 
with $n\not\equiv3\bmod4$, 
admits a totally real embedding into $\C^n$ 
and an embedding into $\R^{2n-2}$, 
then does it admit a CR-regular embedding into $\C^{n-1}$?}
\end{remark}

\section{Generic immersions of simply-connected $5$-manifolds into $\C^4$}\label{sect:5dim}
Let $M^5$ be a closed simply-connected $5$-manifold. 
Then, 
%
by definition, 
\begin{align*}
\hat\chi_{\mathbb{R}}(M^5)
&=\mathrm{rank}\thinspace H_2(M^5;\Z)+1&\pmod2.
\end{align*}
Therefore, Corollary~\ref{cor:imm}, together with 
Slapar--Torres's result \cite[Theorem~C]{Slapar-Torres} 
claiming that 
every closed simply-connected $5$-manifold has 
trivial complexified tangent bundle, 
implies the following. 

\begin{theorem}\label{thm:5dim}
A closed simply-connected $5$-manifold $M^5$ admits 
a CR-regular immersion into $\C^4$ if and only if 
the rank of $H_2(M^5;\Z)$ 
is odd. 
\end{theorem}

\begin{table}\renewcommand{\arraystretch}{1.2}
\[\begin{array}{|c||c|c|c|c|}
\hline
Y&H_2(Y;\Z)&w_2w_3[Y]&\hat\chi_{\mathbb{Z}/2\mathbb{Z}}(Y)&\hat\chi_{\mathbb{R}}(Y)\\
\hline
\hline
M_k\ (k\ge2)&\Z/k\Z\oplus\Z/k\Z&0&1&1\\
\hline
M_{\infty}\cong S^2\times S^3&\Z&0&0&0\\
\hline
X_{-1}\cong SU(3)/SO(3)&\Z/2\Z&1&0&1\\
\hline
X_{0}\cong S^5&0&0&1&1\\
\hline
X_j\ (j\in\N)&\Z/2^j\Z\oplus\Z/2^j\Z&0&1&1\\
\hline
X_{\infty}\cong S^2\mathbin{\tilde{\times}}S^3&\Z&0&0&0\\
\hline
\end{array}\]
\caption{Barden's list}\label{table:barden}
\end{table}

Barden \cite{MR0184241} has proved that 
every closed simply-connected $5$-manifold $M^5$ 
can be expressed as a connected sum of the manifolds 
listed in Table~\ref{table:barden}; more precisely, 
$M^5$ is diffeomorphic to a connected sum 
$M_{k_1}\mathbin{\sharp}\cdots\mathbin{\sharp}M_{k_l}\mathbin{\sharp}X_m$, 
where $-1\leq m\leq\infty,\ l\geq0,\ k_1>1$, and 
$k_i$ divides $k_{i+1}$ or $k_{i+1}=\infty$. 
Since each manifold in Barden's list is irreducible, 
with respect to connected sum, 
except $X_1\cong X_{-1}\mathbin{\sharp}X_{-1}$, 
we may refer to the above decomposition 
just as the connected sum decomposition of $M^5$. 
Then, the following is an easy 
interpretation of Theorem~\ref{thm:5dim}. 

\begin{corollary}\label{cor:5dim}
A closed simply-connected $5$-manifold $M^5$ admits 
a CR-regular immersion into $\C^4$ 
if and only if the sum of the numbers of copies 
of $S^2\times S^3$ and $S^2\mathbin{\tilde{\times}}S^3$, 
appearing in the connected sum decomposition of $M^5$, 
is odd. 
\end{corollary}

It seems interesting to compare Corollary~\ref{cor:5dim} 
with \cite[Theorem~C and Remark~2]{Slapar-Torres}. 

\section{Generic immersions of spheres}\label{sect:spheres}
The basic ideas of the proofs of Propositions~\ref{thm:kasuya} 
and \ref{prop:imm} are applicable 
for CR-regular immersions of spheres in more general 
codimensions. 

We start with a very easy observation. 

\begin{proposition}\label{prop:elg}
Suppose that a closed orientable $n$-manifold $M^n$ admits 
a CR-regular immersion $f\co M^n\to\C^q$ with normal bundle $\mathcal{N}$. 
If the normal bundle $\mathcal{N}$ possesses 
$r$ linearly independent sections,
then $\mathop{\mathrm{span}}M^n\geq r$, 
that is,
$M^n$ admits $r$ everywhere linearly
independent vector fields. 
\end{proposition}

\begin{proof}
A similar argument can be found in the proof of 
Propositions~\ref{thm:kasuya}. 
In each $T_xM^n$, 
we can find a complex $(n-q)$-space, 
isomorphic to $df_x(T_xM^n)\cap Jdf_x(T_xM^n)$ through $df_x$. 
Thus we have a $(n-q)$-dimensional complex bundle 
$\mathcal{L}$ over $M^n$ and 
the decomposition 
\[
TM^n\cong\mathcal{L}_{\mathbb{R}}\oplus\mathcal{M}, 
\]
where $\mathcal{L}_{\mathbb{R}}$ is the underlying real bundle 
of $\mathcal{L}$ and $\mathcal{M}$ is the complementary bundle. 
Furthermore we have 
\[
\mathcal{L}_{\mathbb{R}}\oplus\mathcal{M}\oplus\mathcal{N}
\cong f^*T\R^{2q}=f^*T\C^q, 
\]
where $\mathcal{M}$ and $\mathcal{N}$ are isomorphic
through $J$. Namely a section of $\mathcal{N}$ determines 
a section of $\mathcal{M}$, that is, a vector field on $M^n$. 
\end{proof}

\begin{corollary}\label{cor:even-dim}
If a closed orientable $n$-manifold admits 
a CR-regular embedding into some complex space $\C^q$ then 
it has trivial Euler characteristic. 
\end{corollary}

\begin{proof}
Since the normal bundle of an embedding of a closed 
orientable $n$-manifold $M^n$ into a Euclidean space 
has trivial Euler class (see \cite[Corollary~11.4]{MR0440554}), 
it has a nowhere-zero section. 
By Proposition~\ref{prop:elg}, 
$M^n$ has a 
nowhere-zero vector field and hence 
vanishing Euler characteristic 
by the Poincar\'e--Hopf theorem. 
\end{proof}

Now we recall Adams' fundamental result \cite{MR0139178} on vector fields on spheres. 

Given a positive integer $N$, 
find integers $a$ and $b$ such that 
$N=(2a+1)2^b$. 
Then the Radon--Hurwitz number is defined as 
\begin{align*}
\rho(N)=
\begin{cases}2b+1&b\equiv0\bmod4,\\2b&b\equiv1,2\bmod4,\\2b+2&b\equiv3\bmod4.\end{cases}
\end{align*}

\begin{theorem}[Adams~\cite{MR0139178}]\label{thm:adams}
The maximum number of linearly independent vector 
fields on $S^{N-1}$ is equal to $\rho(N)-1$. 
\end{theorem}

The following generalises Elgindi's theorem \cite[Theorem~2]{1607.02384}. 

\begin{theorem}\label{thm:CR-reg emb}
\footnote{In the published version of this paper, 
the condition ``$n\not\equiv3\bmod4$'' in (c) 
was misstated as ``$n\equiv3\bmod4$''. We have corrected it.}
There does not exist a CR-regular embedding 
of $S^n$ into $\C^q$ 
in each case of the following:
\begin{enumerate}[(a)]
\item
$n\equiv0\bmod2$; 
\item
$4q\geq3n+2$ and $n\ne1,3,7$; 
\item
$2q=n+3>6$ and $n\not \equiv3\bmod4$. 
\end{enumerate}
\end{theorem}

\begin{proof}
\begin{enumerate}[(a)]
\item
This is a direct consequence of Corollary~\ref{cor:even-dim}. 

\item
For $4q\geq3n+2$, 
any embedding of $S^n$ into $\C^q$ has 
trivial normal bundle (Kervaire \cite[Theorem~8.2]{MR0102809}). 
In view of Proposition~\ref{prop:elg} and 
Theorem~\ref{thm:adams}, 
it suffices to show that $2q-n\geq\rho(n+1)$. 

Let $n+1=(2a+1)2^b$ for non-negative integers $a$ and $b$. 
As a result of (a), we may assume that $n$ is odd, that is, $b\geq1$. 
Thus, the condition $4q\geq3n+2$ implies that 
$2q-n\geq\dfrac{n+3}2$. 

If $a\geq1$ we have 
\[
2q-n\geq\dfrac{n+3}2=\dfrac{2a+1}22^b+1\geq2^b+2\geq2b+2\geq\rho(n+1). 
\]
If $a=0$ then we have $b\geq4$ by $n\ne1,3,7$. 
Then if $b=4$, we have $n+1=16$ and hence 
\[
2q-n\geq\dfrac{n+3}2=2^{b-1}+1=9=\rho(16). 
\]
If $a=0$ and $b\geq5$, we see that  
\[
2q-n\geq\dfrac{n+3}2=2^{b-1}+1\geq2b+2\geq\rho(n+1). 
\]

This completes the proof. 

\item
An embedding of $S^n$ into $\R^{n+3}$ has 
trivial normal bundle (Massey \cite[Corollary, p.\,960]{MR0109351}). 
Therefore, by Proposition~\ref{prop:elg}, 
the existence of a CR-regular embedding 
of $S^{4k+1}$ into $\C^{2k+2}$, $k\geq1$, 
would require $S^{4k+1}$ to have 
three linearly independent vector fields 
and hence vanishing real semi-characteristic by 
Atiyah's theorem (see \cite[p.\,108]{MR0287567}). 
This, however, is impossible. 
\end{enumerate}
\end{proof}

\begin{remark}
It would be interesting to know whether 
there exists a CR-regular embedding of $S^n$ 
into $\C^q$ for $(n,q)$ outside the range of 
Theorem~\ref{thm:adams}. 
\end{remark}

\section*{Acknowledgements}
The first-named author has been 
supported in part by the Grant-in-Aid for Young Scientists (B), 
No.~17K14193, Japan Society for the Promotion of Science. 
The second-named author has been 
supported in part by the Grant-in-Aid for Scientific Research (C), 
No.~15K04880, Japan Society for the Promotion of Science. 
%


\end{document}